\documentclass[a4paper, 12pt,pdftex
]{amsart}

\usepackage[all]{xy}
\usepackage[latin1]{inputenc}
\usepackage{amsmath,amssymb,amsfonts, amsthm, amscd, latexsym, marginnote}

\usepackage{rotating}
\usepackage[english]{babel}
\usepackage[mathcal]{euscript}

\usepackage[hyperindex,breaklinks,pdftex]{hyperref} 

\usepackage{graphicx}
\CompileMatrices

\newtheorem{thm}{Theorem}[section]

\newtheorem{prop}[thm]{Proposition}
\theoremstyle{definition}

\theoremstyle{definition}

\newtheorem{rmk}[thm]{Remark}

\newcommand\Q{\mathbb Q}

\newcommand{\CP}{{\mathbb{CP}}}
\newcommand{\Z}{{\mathbb Z}}

\newcommand{\N}[0]{\mathbb{N}}
\newcommand{\K}[0]{\mathbb{K}}

\newcommand{\into}[0]{\hookrightarrow}

\newcommand{\SL}[0]{SL_{2}(\Z)}

\def\qed{\ifmmode $\Box$ \else{\unskip\nobreak\hfil
\penalty50\hskip1em\null\nobreak\hfil $\Box$
\parfillskip=0pt\finalhyphendemerits=0\endgraf}\fi}

\newcommand{\eq}[1][r]
       {\ar@<-3pt>@{->}[#1]
        \ar@<-1pt>@{}[#1]|<{}="gauche"
        \ar@<+0pt>@{}[#1]|-{}="milieu"
        \ar@<+1pt>@{}[#1]|>{}="droite"
        \ar@/^2pt/@{-}"gauche";"milieu"
        \ar@/_2pt/@{-}"milieu";"droite"}
\def\sur{\ar@{>>}[r]}
\newcommand{\imm}[1][r] {\ar@{^{(}->}[#1]}

\newcommand{\coker}[0]{\mathrm{coker}}

\def\H#1{\save
[].[drrr]="g#1"*[F-,]\frm{}\restore}%

\begin{document}


\title[Principal congruence subgroups and geometric representations]{Cohomology of braids, principal congruence subgroups and geometric representations}

\author[F.~Callegaro]{F.~Callegaro}
\address{Scuola Normale Superiore, Pisa, Italy}
\email{f.callegaro@sns.it}

\author[F.~R.~Cohen]{F.~R.~Cohen$^{*}$}
\address{Department of Mathematics,
University of Rochester, Rochester NY 14627, USA}
\email{cohf@math.rochester.edu}
\thanks{$^{~*}$Partially supported by DARPA}

\author[M.~Salvetti]{M.~Salvetti}
\address{Department of Mathematics,
University of Pisa, Pisa Italy}
\email{salvetti@dm.unipi.it}

\date{\today}

\begin{abstract}
The main purpose of this article is to give the integral cohomology 
of classical principal congruence subgroups 
in $SL(2, \mathbb Z)$ as well as their analogues in the 
third braid group with local coefficients in symmetric powers of the 
natural symplectic representation. The resulting answers (1) correspond to 
certain modular forms in characteristic zero, and (2) the cohomology 
of certain spaces in homotopy theory in characteristic $p$. 
The torsion is given in terms of the structure of a ``p-divided 
power algebra''.

The work is an extension of the work in \cite{b3cohom} as
well as extensions of a classical computation of Shimura to integral coefficients.  
The results here contrast the local coefficients such as that in \cite{looij96} and \cite{tillmann10}.


\end{abstract}

\maketitle

\section{Introduction} \   In a previous paper (\cite{b3cohom}) we
considered the ring $M=\mathbb Z[x,y]$ of two variables integral
polynomials as an $\SL$-module: the  action on $M$ is given by
extending to the symmetric algebra the natural  action over $\mathbb Z^2.$
Of course, $M$ decomposes as an
$\SL$-representation into its homogeneous components of degree $n$:
$M=\oplus_{n\geq 0}\ M_n$. We also considered the short exact sequence  
\begin{equation} \label{fibration}
\begin{CD}
1  @>>> \Z  @>j>>  B_3 @>\lambda>> \SL @>>> 1
\end{CD}
\end{equation}
where $\lambda$ is the natural representation of the braid group $B_3$ 
over the homology of a
genus $1$ surface, taking the two standard generators of $B_3$ into the 
automorphisms induced by Dehn twists around one parallel and one meridian
 respectively. The kernel of $\lambda$ is given by twice the center of $B_3.$ 
Therefore,  $M$ is also a $B_3$-module with the induced action.  
In the above cited paper we completely determined the cohomology of $B_3$ and 
of $\SL$ with coefficients in the module $M.$ The answer 
is not at all trivial: the free part of the resulting cohomology is related
 to the dimension of certain spaces of modular functions, while the
 $p$-torsion is expressed in terms of divided  powers rings. 
The intriguing connections of the torsion part with some important topological
 constructions related to homotopy theory is still to be understood. 

It is natural to ask if our methods can be extended to the plethora of
well-known subgroups  
of $\SL.$ In this paper we answer affirmatively to this
question for one of the main classes of subgroups of $\SL,$ the
so called {\it principal congruence subgroups 
of level} $n,$ $\Gamma(n)\subset \SL.$  Recall that $\Gamma(n)$ is
defined as the  kernel of the mod-$n$ reduction map $\SL \to
SL_2(\Z_n).$ We extend our methods to this case, by computing the
cohomology both of $\Gamma(n)$ and of the group $B_{\Gamma(n)}:=\ $ $
\lambda^{-1}(\Gamma(n)),$  with coefficients in the above defined
representations. The case $n=2$ is particularly significant: 
$\Gamma(2)$ is the kernel of the map over the symmetric group
$S_3=SL_2(\mathbb Z_2)$ and $B_{\Gamma(2)}\subset B_3$ is the
pure braid group $P_3$ in three strands.  

With respect to \cite{b3cohom} the computation of the cohomology 
becomes easier from the facts that $\Gamma(n)$ is either a free group of finite rank,
%
%
%
for $n>2,$ or a product of a free group
times $\mathbb Z_2,$ for $n=2.$ Therefore the cohomology of $\Gamma(n)$
is trivial in dimension higher than one (for $n>2$). However, the
description of the first cohomology group is similar to that of $\SL$
(see theorem \ref{H_1gamman}) and is done in terms of a
so called {\it $p-$divided polynomial algebra} described below.   

The complete description of the cohomology of the subgroup
$B_{\Gamma(n)},$ which has cohomological dimension two, is given in
the final theorem \ref{bgamma}.

The work here is an extension of the work in \cite{b3cohom} as
well as extensions of a classical computation of Shimura to integral coefficients.  
The results presented in this paper contrast the local coefficients such as that in \cite{looij96} and \cite{tillmann10}.

\section{Principal congruence subgroups}\label{Principal congruence subgroups}

 It is well known that the subgroup
$\Gamma(n)\subset \SL$  is either a finitely  
generated free group or a product of a finitely generated free group
with $\Z_2$ (see \cite{frasch, gross}).  
This fact follows from the decomposition  $PSL_2(\Z) = \Z_2 * \Z_3,$
so by the Kurosh subgroup  
Theorem (\cite{kurosh})  the kernel of the map
$$
PSL_2(\Z) \to PSL_2(\Z_n)
$$
is always free.
Hence $\Gamma(2)$ is the product of a free group with $\Z_2 = \{I, -I
\},$ while for $n >2$ one gets that 
$\Gamma(n)$ is free.

We write $B_{\Gamma(n)}$ for the inverse image
$\lambda^{-1}(\Gamma(n))$ in $B_3$ (see (\ref{fibration}) in the
introduction).

In the following we show how to extend the
computations in \cite{b3cohom} of the cohomology of $\SL$ and of $B_3$
to that of $\Gamma(n)$ and of $B_{\Gamma(n)}.$  
A special case (as said in the introduction) is $n =2:$
we have the isomorphism $$SL_2(\Z_2) \simeq S_3$$ onto the symmetric
group $S_3,$ while $$B_{\Gamma(2)}\simeq P_3$$ is the pure braid group 
on three strands.

Recall that the group $\SL$ is generated by the elements $$ s_1 =
\left( \begin{array}{cc} 
                    1 & 0 \\ -1 & 1
                   \end{array} \right)
\mbox{ and } s_2 =  \left( \begin{array}{cc}
                    1 & 1 \\ 0 & 1
                   \end{array} \right) $$
and the action on the module $M = \Z[x,y]$ is given by:
$$
s_1: \left\{\begin{array}{l}
             x \mapsto x-y \\ y \mapsto y
            \end{array}
 \right. \quad \quad \quad s_2: \left\{ \begin{array}{ l}
             x \mapsto x \\ y \mapsto x+y.
            \end{array} \right.
$$

We start with some preliminar results.
\begin{prop} \label{p:surj}
Let $m>1$ be an integer. 
Let $G$ be a subgroup of $\SL$ containing the elements $s_1^{b}$ and
$s_2^{c}$ such that $b$ and $c$ 
are coprime with $m$. 
The restriction of the projection $\SL \to SL_2(\Z_{m})$ to the
subgroup $G$ is surjective map.
\end{prop}

\begin{proof} Recall that the projection $\SL \to SL_2(\Z_m)$ is surjective (see, for example, \cite{b3cohom}). As a consequence the proposition is straightforward since $b$ and $c$ are
invertible mod $m$ and hence a suitable power
of $s_1^{b}$ (resp. $s_2^{c}$) maps to $s_1$ (resp. $s_2$) in $SL_2(\Z_{m})$. 
\end{proof}

\begin{prop}\label{p:han}
If the integer $m>1$ factorizes as $m = p_1^{a_1} \cdots p_k^{a_k}$ then 
\begin{equation*} 
SL_2(\Z_m) \simeq SL_2(\Z_{p_1^{a_1}}) \times \cdots \times SL_2(\Z_{p_k^{a_k}}).
\end{equation*}
%
\end{prop}

\begin{proof}
The result is an easy consequence of the Chinese Remainder Theorem (see also \cite{han}). 
\end{proof}

\begin{prop} \label{p:p_does_not_divide}
Let $m$ be a positive integer and let $p$ be a prime that does not divide $m$. For each non-negative integer  $a$ and for any integer $n$ we have 
$$H^0(\Gamma(m);M_n \otimes \Z_{p^a}) = H^0(\SL;M_n \otimes \Z_{p^a}).$$
\end{prop}
\begin{proof}
The result follows from proposition \ref{p:surj} applied to
$s_1^{m}, s_2^{m} \in \Gamma(m)$, since
$$
H^0(\Gamma(m);M_n \otimes \Z_{p^a}) = (\Z_{p^a}[x,y])^{\Gamma(m)} = (\Z_{p^a}[x,y])^{\SL} = H^0(\SL;M_n \otimes \Z_{p^a}).
$$

\vspace{-\baselineskip}
\end{proof}

We need to recall the following result:
\begin{prop}[\cite{steinberg}, Theorem H and following remarks] \label{stei}
Let $p$ be a prime number and let $a, b$ be positive integers, $a \leq b$.  
The module $H^0(\Gamma(p^a);M \otimes \Z_{p^b})$ is generated by
polynomials of the form
$$c P(x^d,y^d)$$ where $c$ is a positive integer such that the power 
$p^{b-a}$ divides the product $c d$.
\end{prop}

We will need a slightly generalized version of proposition \ref{stei}:
\begin{prop}\label{p:stei_gen}
Let $p$ be a prime number and let $a, b$ be positive integers, $a \leq b$.  
Let $m$ be a positive integer and suppose $p^a \mid m$ and $p^{a+1} \nmid m$.
The module $H^0(\Gamma(m);M \otimes \Z_{p^b})$ is generated by
polynomials of the form
$$c P(x^d,y^d)$$ where $c$ is a positive integer such that the power 
$p^{b-a}$ divides the product $c d$.
\end{prop}
\begin{proof}
Our aim is to prove that 
$$H^0(\Gamma(m);M \otimes \Z_{p^b}) = H^0(\Gamma(p^a);M \otimes \Z_{p^b})$$
so that we can reduce our result to the previous proposition.
The equality above is equivalent to
$$ (M \otimes \Z_{p^b})^{\Gamma(m)} = (M \otimes \Z_{p^b})^{\Gamma(p^a)}.
$$
Since the action of $\Gamma(m)$ and $\Gamma(p^a)$ on the module $M \otimes \Z_{p^b}$ 
factors through the action of the group $SL_2(\Z_{p^b})$ 
we need to show that the groups $\Gamma(m)$ and $\Gamma(p^a)$
have the same image with respect to the projection  $\SL \to SL_2(\Z_{p^b})$.

Let us consider the following commutative diagram:
\begin{center}
\begin{tabular}{l}
\xymatrix @R=1.5pc @C=1.5pc {
\Gamma(p^a) \imm & \SL \sur & SL_2(\Z_{p^b}) \\\ar[r]
\Gamma(m) \imm \imm[u] & \SL \sur \ar[u]_{\simeq} & SL_2(\Z_{p^{b-a}m}). \ar[u] \\
}
\end{tabular}
\end{center}

 
Since we assume $a<b$,
the projections $\Gamma(p^a) \to SL_2(\Z_{p^b})$ and $\Gamma(m) \to SL_2(\Z_{p^{b-a}m})$ factor through the quotients
\begin{center}
\begin{tabular}{lcl}
\xymatrix @R=1.5pc @C=1.5pc {
\Gamma(p^a) \ar[rr] \ar@{>>}[dr] & & SL_2(\Z_{p^b}) \\
& \Gamma(p^a)/\Gamma(p^b) \ar[ur] & \\
} & \mbox{and} & 
\xymatrix @R=1.5pc @C=1.5pc {
\Gamma(m) \ar[rr] \ar@{>>}[dr] & & SL_2(\Z_{p^{b-a}m}) \\
& \Gamma(m)/\Gamma(p^{b-a}m) \ar[ur] & \\
}
\end{tabular}
\end{center}
and finally, since $\Gamma(p^b) \cap \Gamma(m) = \Gamma(p^{b-a}m)$, the inclusion $\Gamma(m) \into \Gamma(p^a)$
induces the inclusion $$\Gamma(m)/\Gamma(p^{b-a}m) \into \Gamma(p^a)/\Gamma(p^b).$$
The result follows since $$\mid \! \Gamma(m)/\Gamma(p^{b-a}m) \! \mid = 
\frac{\mid \! SL_2(\Z_{p^{b-a}m}) \! \mid}{\mid \! SL_2(\Z_m) \! \mid} =
\frac{\mid \! SL_2(\Z_{p^b}) \! \mid}{\mid \! SL_2{\Z_{p^a}}\! \mid} = \mid \! \Gamma(p^a)/\Gamma(p^b) \! \mid
$$
(the equality in the middle is a consequence of proposition \ref{p:han}) and hence we have the following commutative diagram
\begin{center}
\begin{tabular}{l}
\xymatrix @R=1.5pc @C=1.5pc {
\Gamma(p^a)/\Gamma(p^b) \ar[r] & SL_2(\Z_{p^b}) \\
\Gamma(m)/\Gamma(p^{b-a}m) \ar[u]_{\simeq} \ar[r] & SL_2(\Z_{p^{b-a}m})\ar[u] \\
}
\end{tabular}
\end{center}
that implies that the groups $\Gamma(m)$ and $\Gamma(p^a)$ have the same image in $SL_2(\Z_{p^b})$.
\end{proof}

Let us consider the case $\Gamma(2)$. The group $\Gamma(2)$ is
the direct product of a free group $F_2$ generated
by the elements $$ s_1^2 = \left( \begin{array}{cc}
                    1 & 0 \\ -2 & 1
                   \end{array} \right)
\mbox{ and } s_2^2 =  \left( \begin{array}{cc}
                    1 & 2 \\ 0 & 1
                   \end{array} \right) $$
and the cyclic group $\Z_2$ generated by 
$$ w_2 = (s_1 s_2 s_2)^2 =  \left( \begin{array}{cc} -1 & 0 \\ 0 &
 -1 \end{array} \right). $$

\begin{prop} \label{p:proj2}
The map $\pi_2: \Gamma(2) \to \Z_2$
defined by $$
\Gamma(2) \ni A = \left( \!\! \begin{array}{cc}
                     \! a_{11} \!& \!a_{12}\! \\ \!a_{21}\! & \!a_{22}\!
                     \end{array} \!\!
 \right) \stackrel{\pi_2}{\mapsto} a_{11} \in \Z_4^* \simeq \Z_2.
$$
is the projection homomorphism onto the second component of the decomposition $\Gamma(2) = F_2 \times \Z_2$ described above.
\end{prop}
\begin{proof}
Since the matrix $A=(a_{ij})$ is in $\Gamma(2)$ we have that $a_{11}$ is invertible $\mod 4$. 
We verify that the map defined in the proposition is a group homomorphism: in fact given $A= \left( \!\! \begin{array}{cc} \!a_{11}\! &\! a_{12}\! \\\! a_{21}\! &\! a_{22} \!\end{array} \!\! \right)$ and $B= \left( \!\! \begin{array}{cc} \!b_{11}\! & \!b_{12}\! \\ \!b_{21}\! &\! b_{22} \!\end{array} \!\! \right)$ two elements in $\Gamma(2)$ we have that 
$
a_{11}b_{11} \equiv a_{11}b_{11} + a_{12}b_{21} \mod 4
$
and hence 
$
\pi_2(A)\pi_2(B) = \pi_2(A\cdot B).
$
Finally it is easy to verify that $F_2 = \ker \pi_2$.
\end{proof}

We obtain a resolution for a $\Gamma(2)$-module $N$ as in the 
following diagram taking the product of the standard periodic 
resolution for $\Z_2$ (horizontal lines) and the minimal resolution 
for $F_2$ (vertical lines):

\begin{equation}\label{resolution}
\begin{gathered}
\xymatrix @R=1.5pc @C=1.5pc {
N \oplus N\ar[r]^{w_2-I} & N \oplus N \ar[r]^{w_2+I}& N \oplus N \ar[r]^{w_2-I}& N \oplus N \ar[r]^{w_2+I}& \cdots \\
N \ar[u]^{(s_1^2-I,s_2^2-I)} \ar[r]_{w_2-I}& N \ar[u]^{(s_1^2-I,s_2^2-I)} \ar[r]_{w_2+I}& N \ar[u]^{(s_1^2-I,s_2^2-I)} \ar[r]_{w_2-I}& N \ar[u]^{(s_1^2-I,s_2^2-I)} \ar[r]_{w_2+I}& \cdots
} \\
\end{gathered}
\end{equation}
Recall that the generator
$w_2$ acts trivially on even degree polynomials, while it acts as $(-1)$-multiplication on odd degree polynomials. 
Applying this observation to the previous diagram we get the following result:
\begin{prop} \label{p:Gamma2} Let $F_2$ be the subgroup of $\SL$ freely generated by $s_1^2, s_2^2$.
The following isomorphisms hold. For even $n$:
$$
H^0(\Gamma(2); M_n) = H^0(F_2; M_n)= \left\{ \begin{array}{l} \Z
  \mbox{ if $n=0;$}\\ 
0 \mbox{ if $n>0;$}\end{array}\right.
$$
$$
H^1(\Gamma(2); M_n) = H^1(F_2; M_n),
$$
and for $i>0$
$$ 
H^{2i}(\Gamma(2); M_n) = H^0(F_2; M_n \otimes \Z_2) = M_n \otimes \Z_2,
$$
$$ 
H^{2i+1}(\Gamma(2); M_n) = H^1(F_2; M_n) \otimes \Z_2.
$$
For odd $n$ 
$$
H^0(\Gamma(2);M_n) = H^0(F_2; M_n) = 0,
$$
and for $i>0$
$$
H^{2i-1}(\Gamma(2); M_n) = H^0(F_2; M_n \otimes \Z_2) = M_n \otimes \Z_2,
$$
$$ 
H^{2i}(\Gamma(2); M_n) = H^1(F_2; M_n) \otimes \Z_2 = H^1(F_2;M_n \otimes \Z_2).
$$
Moreover for any $n$ we have
$$
 H^1(F_2;M_n \otimes \Z_2) = (M_n \oplus M_n) \otimes \Z_2.
$$
%
\end{prop}
\begin{proof}
The proposition follows from an analysis of the resolution of $\Gamma(2)$ given in (\ref{resolution}). The last statement about the cohomology of $F_2$ follows since the action of the group $F_2$ on the module $M \otimes \Z_2$ is trivial.
\end{proof}

%
%
\begin{prop} \label{p:rankH}
The group $H^0(\Gamma(m);M_n)$ is isomorphic to $M_0$ for $n = 0$ and
is trivial for $n > 0$.
Let $p>2$ be a prime number. The group $H^1(\Gamma(p);M_n \otimes \Q)$
has rank $1+p(p^2-1)/12$ for $n=0$ and $(n+1)p(p^2-1)/12$ for
$n>0$. For $p=2$ we have the isomorphism $H^1(\Gamma(2); M_n \otimes
\Q)= \Q^2,$ concentrated in degree $0$.
\end{prop}
\begin{proof}
The first statement is trivial, since we have that 
$s_1^m, s_2^m \in \Gamma(m)$.
The second statement follows from a result of Frasch (\cite{frasch}): 
for $p>2$ the group $\Gamma(p)$ is free on
$N(p)=1+p(p^2-1)/12$ generators.
We apply this result using the minimal resolution of the free group
$\Gamma(p)$. The computation
of the rank follows since the submodule of invariants in $M_n$ is
trivial for $n>0$ and hence the differential $\delta_0: M_n \to
M_n^{N(p)}$ is injective.
The case $p=2$ can be easily checked separately by using the formulas
from proposition \ref{p:Gamma2}.
\end{proof}

\begin{rmk}
We can compute the rank of $\Gamma(m)$ as a free group for any $m$ ($m > 2$).
Let $m$ be an integer, $m>2$. Suppose that $m$ is not prime and let
$p$ be a prime that divides $m$.
Then we have that $\Gamma(m) \subset \Gamma(p)$. Let $i$ be the index 
$$
i = [\Gamma(p): \Gamma(m)] = \mid\! SL_2(\Z_m)\! \mid / \mid\!
SL_2(\Z_p)\! \mid.
$$
If $p \neq 2$ the group $\Gamma(p)$ is free and the Schreier index formula
(\cite{schreier}) gives us that $\Gamma(m)$ is a group of rank
$$
r = \mathrm{rk} \Gamma(m) = i (\mathrm{rk}  \Gamma(p)-1) +1.
$$
In the case $p=2$ the group $\Gamma(2)$ is not free but it contains a
rank-$2$ free subgroup $F_2$ of index $2$. 

If $m \equiv 0 \mod 4$ then $\pi_2(\Gamma(m))=1$ (see \ref{p:proj2}) so  
$\Gamma(m) \subset F_2$.
Hence the Schreier formula become:
$$
r = \mathrm{rk} \Gamma(m) = i/2 +1.
$$

Now let us consider the case $m \equiv 2 \mod 4$ we claim that $\Gamma(m)$ is not contained in $F_2$.
In fact we can consider the element $U_m := \left( \! \! \begin{array}{cc} \! m+1 \! & \! -m\! \\\! m\! & \!-m+1\! \end{array} \!\!\right) \in \Gamma(m)$ and we can see that $\pi_2(U_m) = -1 \mod 4$. In order to determine the rank 
of $\Gamma(m)$ one can consider the inclusion $\Gamma(m) \cap F_2 \subset F_2$, of index $i = [ \Gamma(2):\Gamma(m)]$ and the inclusion $\Gamma(m) \cap F_2 \subset \Gamma(m)$, of index $2$.
From the first inclusion Schreier formula tells us that $\Gamma(m) \cap F_2$ is a free group of rank $i+1$.
From the second inclusion Schreier formula gives
$$
\mathrm{rk} (\Gamma(m)\cap F_2) = 2 (\mathrm{rk} \Gamma(m) -1) + 1
$$
and then $\mathrm{rk} \Gamma(m) = i/2 +1.$

As a consequence we can use the same argument of Theorem \ref{p:rankH}
and compute the dimension of $H^1(\Gamma(m);M_n \otimes \Q)$ for any
$m$ using the following formulas: $$ \mid\! SL_2(\Z_p) \! \mid =
(p+1)(p-1)p$$ and (see for example \cite{han}) $$ \mid\! SL_2(\Z_{p^a}) \! \mid =
p^{(a-1)3}(p+1)p(p-1).$$
\end{rmk}
From the previous remark and from proposition \ref{p:han} we have the following generalization of 
proposition \ref{p:rankH}:
\begin{prop}
Let $m>2$ be an integer that factors as $m = p_1^{a_1} \cdots p_k^{a_k}$. The cardinality of $SL_2(\Z_m)$ is 
given by 
$$
d= \prod_i p_i^{(a_i-1)3}p_i(p_i^2-1)
$$
and if we define $i = \frac{d}{p_1(p_1^2-1)}$ then $\Gamma(m)$ is a free group of rank 
$$r = \left\{ \begin{array}{ll}
i/2 + 1 & \mbox{ if }p_1 = 2 \\
i(p_1(p_1^2-1)-1)+1  & \mbox{ if }p_1 > 2            \end{array}
\right.$$
The rank of the group $H^1(\Gamma(m); M_n \otimes \Q)$ is $r$, for $n=0$ and $(r-1)(n+1)$ for $n>0$.
\qed 
\end{prop}
\begin{rmk}
In the factorization of $m$ used in proposition \ref{p:rankH} the primes $p_i$'s need not be ordered, so any prime factor of $m$ can be used in the formula to compute the rank of $\Gamma(m)$.
\end{rmk}


We now consider the torsion part of the cohomology. We need to recall
the notion of {\it divided polynomial algebra}, as defined in \cite{b3cohom}. We present a 
slightly more general definition.


Let $\Q[x]$ be the ring of rational polynomials in one variable. One
defines the subring $$\Delta[x]:= \Delta_{\Z}[x] \subset \Q[x]$$  
as the subset generated, as a $\Z$-submodule, by the elements
$x_n:=\ \frac{x^n}{n!}$ for $n \in \N$.

\medskip
\noindent {\bf Notational remark.} 
The most frequent notation for this ring uses the letter $\Gamma$
instead of $\Delta;$ for obvious reasons, we decided to change it to
avoid confusion with the principal congruence subgroups. 

\medskip
It follows from
\begin{equation}\label{dividedpoly}
 x_i x_j\  =\   \binom{i+j}{i} x_{i+j}
 \end{equation}
that $\Delta[x]$ is a sub-algebra of $\Q[x],$
usually known as the \emph{divided  polynomial algebra} 
over $\Z$. 
One can define
$\Delta_R[x]:=\Delta[x]\otimes R$ over any ring $R.$

Let $p$ be a prime number and $p^a$ a power of $p$. Consider the $p-$adic valuation
\ $v:=v_p:\N\setminus\{0\}\to \N$ such that $p^{v(n)}$ is the maximum
power of $p$ dividing $n.$  

 Define the ideal  $I_{p^a}$  of $\Delta[x]$ as
$$I_{p^a}:= \ (p^{v(i)+a}\ x_i,\quad i\geq 1)$$
and call the quotient  $$ \Delta_{p^a}[x] := \Delta[x]/I_{p^a}$$
the \emph{ ${p^a}$-local divided polynomial algebra}.

The following two propositions were proven in \cite{b3cohom} for $a=1.$ Exactly the same proof works for any $a\geq 1.$

\begin{prop}\label{prop:algebra}
The ring $\Delta_{p^a}[x]$ is naturally isomorphic to the
quotient $$\Z[\xi_1,\xi_p, \xi_{p^2}, \xi_{p^3}, \ldots]/J_{p^a}$$ where
$J_{p^a}$ is the ideal generated by the polynomials  $$ p^a\xi_1\ , \quad 
\xi_{p^i}^p - p \xi_{p^{i+1}}\  (i \geq 1 ).$$
The element $\xi_{p^i}$ corresponds to the generator $x_{p^{i}} \in
\Delta_{p^a}[x],\ i\geq 0$. 
\qed\end{prop}
\bigskip

\begin{prop}
Let $\Z_{(p)}$ be the local ring obtained by inverting numbers prime
to $p$ and let  
$\Delta_{\Z_{(p)}}[x]$ be the divided polynomial algebra over $\Z_{(p)}.$ 
 One has an isomorphism:
 $$\Delta_{p^a}[x]\ \cong\ \Delta_{\Z_{(p)}}[x]/({p^a}x).$$
\qed\end{prop}

Notice that if \ $\Delta[x]$ \ is graded with $\deg x = k$, then 
$\Z[\xi_1, \xi_{p}, \xi_{p^2}, \ldots]/J_{p^a}$ is graded with $\deg
\xi_{p^i} = k p^i$.

A multi-variable divided polynomial algebra is defined as 
$$
\Delta[x,x',x'', \ldots] := \Delta[x] \otimes_\Z \Delta[x'] \otimes_\Z
\Delta[x'']\otimes_\Z \cdots 
$$
with the ring structure induced as subring of
$\Q[x,x',x'',\ldots]$. We have also
$$
\Delta_{p^a}[x,x',x'', \ldots] := \Delta_{p^a}[x] \otimes_\Z \Delta_{p^a}[x']\otimes_\Z
\Delta_{p^a}[x''] \otimes_\Z \cdots. 
$$

 We also define the submodule   
$$ {\Delta}^+_{p^a}[x, x', x'', \ldots] := {\Delta}_{p^a}[x, x', x'',
  \ldots]_{\deg > 0}.$$


We recall from \cite{b3cohom} the following desctiption of the first cohomology group of $\SL$:

\begin{thm}[\cite{b3cohom} Theorem 3.7]
Let $M=\oplus_{n\geq 0}\ M_n \simeq H^*((\CP^\infty)^2;\Z)$ be the $\SL$-module already defined. Then the cohomology $H^1(\SL;M_n)$ is given as follows:
\begin{enumerate}
\item the $p$-torsion component is given by $$H^1(\SL;M_n)_{(p)} = 
\Delta^+_p[\mathcal P_p, \mathcal Q_p]_{\deg = n}$$
where we fix gradings $\deg \mathcal P_p = 2(p+1)$ and 
$\deg \mathcal Q_p = 2p(p-1)$;
\item the free part is $$FH^1(\SL;M_n) = \Z^{f_n}$$ for $n > 0$ 
where the rank $f_n$ is given by the Poincar\'e series
\begin{equation*}
P^1_{\SL,0}(t)
= \sum_{n=0}^\infty f_n t^n = 
\frac{t^4(1+t^4-t^{12}+t^{16})}{(1-t^8)(1-t^{12})}. 
\end{equation*}
\end{enumerate}
\end{thm}

We come back now to the torsion part of the cohomology of the
principal congruence subgroups $\Gamma(n).$ 

By means of the Universal Coefficients Theorem we can compute
the $p$-torsion part of the first cohomology group  $H^1(\Gamma(m);M_n)$
by comparing it with the group $H^0(\Gamma(m);M_n\otimes \Z_{p^a})$.
This is described in Propopsition \ref{p:p_does_not_divide} in the case of a 
prime $p$ that does not divide $m$ and in Proposition \ref{p:stei_gen} in the case
that $p$ divides $m$. 

\begin{thm}\label{H_1gamman}
Let $p$ be a prime number and $m>1$ an integer. 
If $p$ is a prime that does not divides $m$ the $p$-torsion 
component of $H^1(\Gamma(m);M_n)$ is given by:
$$ H^1(\Gamma(m);M_n)_{(p)} = H^1(\SL;M_n)_{(p)}  = 
\Delta^+_p[\mathcal P_p, \mathcal Q_p]_{\deg = n}$$
where $\mathcal P_p$ and $\mathcal Q_q$ are elements of degree  $\deg \mathcal P_p = 2(p+1)$ and 
$\deg \mathcal Q_p = 2p(p-1)$.
If $p$ divides $m$, suppose $p^a \mid m, p^{a+1} \nmid m$. Then we have 
$$
H^1(\Gamma(m);M_{>0})_{(p)} \simeq \Delta_{p^a}^+[x,y]
$$
where $x,y$ are elements of degree $1$. \qed
\end{thm}

\begin{rmk}
For $m >2$, since $\Gamma(m)$ is free and hence it has homological dimension $1$, a complete description of the cohomology of $\Gamma(m)$ comes directly from theorem \ref{H_1gamman}.
For $m=2$ the description of the cohomology of $\Gamma(m)$ follows from  theorem \ref{H_1gamman} and proposition \ref{p:Gamma2}. 
\end{rmk}

Now we focus on the description of the cohomology of the groups $B_{\Gamma(m)}$
defined in the introduction.

The Serre spectral sequence for the extension $\Z \to B_{\Gamma(2)}
\to \Gamma(2)$ is concentrated on the first two lines and the
$E_2$-term is the following:

\begin{center}
\begin{tabular}{l}
\xymatrix @R=1pc @C=1pc {
\H1 H^0(\Gamma(2),M) \ar[rrd]^(.7){d_2} & H^1(\Gamma(2),M)
\ar[rrd]^(.7){d_2} & H^2(\Gamma(2),M) &  \cdots \\
H^0(\Gamma(2),M) & H^1(\Gamma(2),M) & H^2(\Gamma(2),M) &  \cdots
} \\
\end{tabular}
\end{center}
\medskip
hence, for $n$ even we get:
\medskip

\def\K#1{\save
[].[drrrr]="g#1"*[F-,]\frm{}\restore}%
\begin{center}
\begin{tabular}{l}
\xymatrix @R=1pc @C=1pc {
\K1 H^0(F_2,M_n) \ar[rrd]^(.7){d_2} & H^1(F_2,M_n) \ar[rrd]^(.7){d_2}
& H^0(F_2,M_n\otimes \Z_2) & H^1(F_2,M_n)\otimes \Z_2 &  \cdots \\
H^0(F_2,M_n) & H^1(F_2,M_n) & H^0(F_2,M_n\otimes \Z_2) &
H^1(F_2,M_n)\otimes \Z_2 &  \cdots
} \\
\end{tabular}
\end{center}
\medskip
and for $n$ odd:
\medskip

\begin{center}
\begin{tabular}{l}
\xymatrix @R=1pc @C=1pc {
\K1 0 & H^0(F_2,M_n\otimes \Z_2) \ar[rrd]^(.7){d_2} &
H^1(F_2,M_n)\otimes \Z_2 & H^0(F_2,M_n\otimes \Z_2) &  \cdots \\
0 & H^0(F_2,M_n\otimes \Z_2) & H^1(F_2,M_n)\otimes \Z_2 &
H^0(F_2,M_n\otimes \Z_2 )&  \cdots
} \\
\end{tabular}
\end{center}
\medskip

Recall that the group $B_{\Gamma(2)} = P_3$ has homological dimension $2$ and hence the spectral 
sequence above converges to $E_{\infty}^{i,j}=0$ for $i+j>2.$ 

For $n$ even we can compute the first differential
$d_2^{0,1}:E^{0,1}_2 \to E_2^{2,0}$ comparing the spectral
sequence above with the spectral sequence for the extension $\Z \stackrel{2}{\to}\Z \to \Z_2$
associated to the subgroup 
$\Z_2 \subset \Gamma(2) = F_2 \times \Z_2$. It follows that the differential $d_2^{0,1}$ is the map
$$M_0^{\Gamma(2)} \to (M_0 \otimes \Z_2)^{\Gamma(2)} $$
induced by the projection map $M_0 \to M_0 \otimes \Z_2$. Hence $\Z
\simeq \ker d_2^{0,1} = 2\Z \subset M_0 = \Z$ and
$\coker d_2^{0,1} = M_{>0} \otimes \Z_2$.
The differential $d_2^{1,1}:E^{1,1}_2 \to E_2^{3,0}$ is surjective. 
All the other maps $d_2:E^{k,1}_2 \to E_2^{k+2,0}$ are isomorphisms.

For odd $n$ the first differential $d^{0,1}_2:E^{0,1}_2 \to E_2^{2,0}$ is
zero and all the other maps $d_2^{k,1}:E^{k,1}_2 \to E_2^{k+2,0}$ are
isomorphisms. So $E_{\infty}^{i,j}=0$ except for $(i,j)=(1,0)$ or $(i,j)=(2,0).$    

For $n$ even the non-zero $E_\infty$ terms of the spectral sequence
with total degree $2$ are 
$$
E^{1,1}_\infty = 2 (\Delta^+_2[x,y])_{\deg = n}
$$
and
$$
E^{2,0}_\infty = M_{n} \otimes \Z_2 = (\Delta^+_2[x,y] \otimes \Z_2)_{\deg = n}.
$$
We claim that there is an extension such that $H^2(B_{\Gamma(2)};M)
\simeq H^1(\Gamma(2);M).$ 
This can be proved using methods similar to those used in
\cite{b3cohom}. We sketch here the main ideas of the proof. 
We can consider the spectral sequence $\overline E$ for
$\Z \to B_{\Gamma(2)} \to \Gamma(2)$ 
with coefficients in the module $M \otimes \Z_{2^k}$ for any $k$. 
Let $\Z_{2^i}$ be a module that appears as a direct summand in
$H^1(F_2, M_n) = E^{1,1}_2$. 

If $k < i$ then  
by Universal Coefficients Theorem the module $\Z_{2^i}$ determines a
module $\Z_{2^k}$ that is a direct summand  
in $H^0(F_2, M_n\otimes \Z_{2^k}) = \overline{E}^{0,1}_2$; moreover
the differential $\overline{d}_2$ 
restricted to $\Z_{2^k}$ is non trivial, as one can see comparing the spectral sequence
$\{\overline{E}_r\}$ with the one associated to the extension $\Z \stackrel{2}{\to}\Z \to \Z_2$.

If $k \geq i$ then by Universal Coefficients Theorem the module
$\Z_{2^i}$ determines a module $\Z_{2^i}$ that is a direct summand  
in $H^0(F_2, M_n\otimes \Z_{2^k}) = \overline{E}^{0,1}_2$; the study
of the Bockstein map $\beta_k$ 
for the extension $0 \to \Z_{2^k} \to \Z_{2^{2k}} \to \Z_{2^k}\to 0$ shows that the
differential $\overline{d}_2$ 
restricted to $\Z_{2^k}$ is non-trivial and the image
$\overline{d}_2(\Z_{})$ is a direct summand of $\overline E_2^{2,0}$
isomorphic to $\Z_2$. 

This means that the summand $\Z_{2^i}$ in $E^{1,1}_2$ give a contribution to 
the cardinality 
of $\mid \!  \! \oplus_{i+j=2}\overline{E}^{i,j}_\infty \! \! \mid = $
$ =  \mid \!  \! H^2(B_{\Gamma(2)};M \otimes \Z_{2^k}) \! \! \mid$
corresponding to a factor $2^{k-1}$ for $k < i$ and to 
a factor $2^{i-1}$ for $k \geq i$.  
It follows by an argument similar to that used in \cite[section 7]{b3cohom}
that the term $\Z_{2^i}$ in $H^1(F_2, M_n) =
E^{1,1}_2$ corresponds to a term 
$\Z_{2^{i-1}}$ in ${E}_\infty^{1,1}$ and there is a summand $\Z_2$ in  
${E}_\infty^{2,0}$ which produce a non-trivial extension
$$
0 \to \Z_2 \to \Z_{2^i} \to \Z_{2^{i-1}} \to 0
$$
in the exact sequence 
$$
0 \to E^{2,0}_\infty \to H^2(B_{\Gamma(2)};M) \to E^{1,1}_\infty \to 0.
$$
Hence there is a direct summand $\Z_{2^i}$ in $H^2(B_{\Gamma(2)};M)$ 
associated to the direct summand 
$\Z_{2^i} \subset H^1(F_2, M_n)$. 


%

%

For $m \neq 2$, as we already said, the group $\Gamma(m)$ is free. 
Hence its cohomology is trivial in dimension bigger than $1$ and it
follows that the spectral sequence for the extension $\Z \to
B_{\Gamma(m)} \to \Gamma(m)$ collapses at 
the $E_2$-term. There is no non-trivial 
extension involved in the $E_{\infty}$-term, since 
the cohomology group $H^*(\Gamma(m);M_0)$ is torsion free and 
for $n>0$ the cohomology
$H^*(\Gamma(m);M_n)$ is concentrated in dimension $1$.

Hence we can describe the cohomology of $B_{\Gamma(m)}$ 
as follows:
\begin{thm}\label{bgamma}
The cohomology of the group $B_{\Gamma(m)}$ with
coefficients in the module $M = \Z[x,y]$ is given as follows:
\begin{itemize}
\item[a)] for $m=2$ and $n=0$  
$$
H^0(B_{\Gamma(2)}; M_0) = \Z,
$$
$$
H^1(B_{\Gamma(2)}; M_0) = \Z^3
$$
$$
H^2(B_{\Gamma(2)}; M_0) = \Z^2
$$
and all the others cohomology groups are zero;
\item[b)] for $m=2$ and even $n>0$ we have 
$$
H^0(B_{\Gamma(2)}; M_n) = H^0(\Gamma(2); M_n),
$$
$$
H^1(B_{\Gamma(2)}; M_n) = H^2(B_{\Gamma(2)}; M_n) = H^1(\Gamma(2); M_n)
$$
and all the others cohomology groups are zero;
\item[c)] for $m=2$ and odd $n$ 
$$
H^1(B_{\Gamma(2)}; M_n) = H^1(\Gamma(2); M_n) = M_n \otimes \Z_2,
$$
$$
H^2(B_{\Gamma(2)}; M_n) = H^2(\Gamma(2); M_n) = (M_n \oplus M_n) \otimes \Z_2
$$
and all the others cohomology groups are zero;
\item[d)] for any $m > 2$ 
$$
H^*(B_{\Gamma(m)}; M_n) = H^*(\Gamma(m); M_n) \otimes 
H^*(\Z; \Z).
$$
\end{itemize}

\vspace{-\baselineskip}
\qed
\end{thm}

\providecommand{\bysame}{\leavevmode\hbox
to3em{\hrulefill}\thinspace}
\providecommand{\MR}{\relax\ifhmode\unskip\space\fi MR }
\providecommand{\MRhref}[2]{%
  \href{http://www.ams.org/mathscinet-getitem?mr=#1}{#2}
} \providecommand{\href}[2]{#2}

\bibliographystyle{amsalpha}
\bibliography{biblio}


\end{document}